\documentclass[11pt,a4paper]{amsart}
\usepackage{amsmath,amsfonts,amsthm,amsopn,color,amssymb,enumitem}
\usepackage{palatino}
\usepackage{graphicx}
\usepackage[english]{babel}
\usepackage{caption}
\usepackage{subcaption}
\usepackage[colorlinks=true]{hyperref}
\hypersetup{urlcolor=blue, citecolor=red, linkcolor=blue}
\usepackage[utf8]{inputenc}
\usepackage{colortbl}
\usepackage{relsize}
\usepackage{esint}
\usepackage[title]{appendix}
\usepackage{enumitem}					
\usepackage[colorinlistoftodos]{todonotes}

\numberwithin{equation}{section}

\newtheorem{theorem}{Theorem}[section]
\newtheorem{lemma}[theorem]{Lemma}

\newtheorem{proposition}[theorem]{Proposition}
\newtheorem{corollary}[theorem]{Corollary}

\theoremstyle{definition}
\newtheorem{remark}[theorem]{Remark}

\newtheorem{definition}[theorem]{Definition}

\newcommand{\LL}{M}

\newcommand{\e}{\varepsilon}

\newcommand{\R}{\mathbb{R}}

\newcommand{\RN}{{\mathbb{R}^N}}
\newcommand{\RD}{{\mathbb{R}^2}}

\newcommand{\N}{\mathbb{N}}

\newcommand{\beq }{\begin{equation}}
\newcommand{\eeq }{\end{equation}}

\DeclareMathOperator{\divv}{div}

\newcommand{\abs}[1]{\lvert{#1}\rvert}

\makeatletter
\@namedef{subjclassname@2020}{\textup{2020} Mathematics Subject Classification}
\makeatother

\begin{document}

\title[Rigidity results for stationary 2D Euler equations]{Rigidity results for finite energy solutions to the stationary 2D Euler equations}

\author{Fabio De Regibus}
\address{ \vspace{-0.4cm}
	\newline 
	\textbf{{\small Fabio De Regibus}} 
	\vspace{0.15cm}
	\newline \indent  
	Departamento de An\'alisis Matem\'atico, Universidad de Granada, 18071 Granada, Spain}
\email{fabioderegibus@ugr.es}

\author{Francesco Esposito}
\address{ \vspace{-0.4cm}
	\newline 
	\textbf{{\small Francesco Esposito}} 
	\vspace{0.15cm}
	\newline \indent  
	Dipartimento di Matematica e Informatica, Università della Calabria, 87036 Rende (Cosenza), Italy}
\email{francesco.esposito@unical.it}

\author{David Ruiz}
\address{ \vspace{-0.4cm}
	\newline 
	\textbf{{\small David Ruiz}} 
	\vspace{0.15cm}
	\newline \indent  
	IMAG, Departamento de An\'alisis Matem\'atico, Universidad de Granada, 18071 Granada, Spain}
\email{daruiz@ugr.es}

\thanks{F.D.R. has been supported by:
	the Juan de la Cierva fellowship Grant JDC2022-048890-I funded by MICIU/AEI/10.13039/501100011033 and by the “European Union NextGeneration EU/PRTR” and he is member of INdAM. F. E. is partially supported by PRIN 2022 project P2022YFAJH {\em Linear and Nonlinear PDE's: New directions and Applications} and he is partially supported by INdAM - GNAMPA Project 2025 - CUP $E5324001950001$.
	 D. R. has been supported by:
	the Grant PID2021-122122NB-I00 funded by MCIN/AEI/ 10.13039/501100011033 and by
	``ERDF A way of making Europe'';
	the Research Group FQM-116 funded by J. Andalucia;
	the \emph{IMAG-Maria de Maeztu} Excellence Grant CEX2020-001105-M/AEI/10.13039/501100011033 funded by MICIN/AEI}



\begin{abstract}
	In this paper we prove rigidity results for classical solutions to the stationary 2D Euler equations in $\RD$. Assuming that the velocity field has finite energy and that the stagnation set is connected, we prove that the corresponding stream function solves an autonomous semilinear elliptic equation. Under some extra conditions on the vorticity near infinity we can also prove that the streamlines are concentric circles. The proofs include several energy estimates on the behavior of the stream function at infinity, as well as an adaptation of the continuous Steiner symmetrization to our setting.

\end{abstract}

\maketitle


\section{Introduction}

In recent years a lot of work has been devoted to the study of  solutions to the stationary Euler equations:
\begin{equation}
	\label{Eu}\tag{E}
	\begin{cases}
		v\cdot\nabla v=-\nabla p&\text{in }\RD,\\
		\divv v=0&\text{in }\RD.
	\end{cases}
\end{equation}
Here $v=(v_1,v_2)$ is the velocity vector field of an inviscid and incompressible fluid and $p$ is the pressure, which is determined up to a constant. When the problem is posed in a domain $\Omega \subseteq \R^2$, tangential boundary conditions are usually assumed, that is, $ v(x) \cdot \nu(x)=0$, where $\nu$ is the unit exterior normal vector to $\partial \Omega$.

Since $\divv v =0$, one can define the stream function $u(x)$ so that
\[
v = (\nabla u)^\perp,
\]
where $(x_1, x_2)^\perp =(- x_2, x_1)$. In this way the streamlines (or trajectories of the fluid) correspond to level curves of the stream function $u$. 

Another important object is the vorticity of the fluid, that in dimension 2 reduces to a scalar:
\[
\omega=\mathrm{curl}(v)=\partial_{x_1}v_2-\partial_{x_2}v_1 = \Delta u.
\]
It is well known that the Euler equations~\eqref{Eu} can be written equivalently as:
\begin{equation}
	\label{Eu2} \tag{E2}
	\begin{cases}
		\nabla B + \omega \, v^{\perp}=0&\text{in }\RD,\\
		\divv v=0&\text{in }\RD,
	\end{cases}
\end{equation}
where $B$ is the Bernoulli function defined as: 
\[
B= \frac 1 2 |v|^2 + p.
\]
In this paper we are concerned with rigidity results for classical solutions, that is, solutions $v \in \mathcal C^1$ satisfying~\eqref{Eu} pointwise. There are several works in the literature on this topic; below we present an overview of the ones more closely related to this paper. In~\cite{HN19} it is shown that, under the assumption:
\begin{equation} \label{HN}  
	0<c < |v(x) |  < C \quad \text{for all } x \in \R^2,
\end{equation}
then $v$ is a shear flow, that is, the streamlines are straight lines. The same result holds also in strips and half-spaces, see~\cite{HN17}. It is quite remarkable that these results are no longer true if~\eqref{HN} is relaxed to:
\[
	0< |v(x) |  < C \quad \text{for all } x \in \R^2,
\]
as has been recently shown in~\cite{DRR25}.	Other results in this line are~\cite{HK25, LWX23, LLSX23}.

In other circumstances one can prove that the flow is circular, that is, the streamlines are concentric circles. In this regard see~\cite{HN23, WZ23} where the problem is posed in annuli or disks: overdetermined cases are also considered there. In~\cite{GPSY21} the case of entire solutions with nonnegative vorticities is treated. See also~\cite{Ru23} for compactly supported solutions, and~\cite{CFZ24} for the problem in exterior domains.

On the other hand, another line of research on this topic is that of flexibility results. An important question here is the existence of solutions close to given ones: see for instance~\cite{CDG21, CZEW23} (see also Open Question 3 in the survey paper~\cite{DE23}). Solutions in the strip which are not shear flows have been found in~\cite{DRR25, GXX24}, whereas nonradial compactly supported solutions have been found in~\cite{GPJ21}, in the vortex patch setting, and in~\cite{EFR24}, with $\mathcal C^k$ regularity (in this regard, see also~\cite{EFRS25}).

\medskip 

Generally speaking, to obtain rigidity results one usually needs some control on the stagnation set:
\[
\mathbf S = \{x\in\RD:v(x)=0\}.
\]

The reason is that the behavior of $v$ around the stagnation points can be quite intricate. Observe for instance that~\eqref{HN} implies that $\mathbf S = \emptyset$. In other cases $\mathbf S$ is assumed to be a point (see~\cite{HN23}), or to have regular boundary (see~\cite{CFZ24, Ru23, WZ23}). In~\cite{DN2024} the authors give rigidity results on \emph{laminar} solutions on periodic channels and annuli, whereas in~\cite{GXX24} assumptions on the \emph{flow angles} are made: in both cases one has a control on the behavior of $v$ near the stagnation set. 
The general idea (that we also exploit in this paper) is to prove that the stream function satisfies a semilinear elliptic equation in the form:
\begin{equation} \label{semilinear}  
	- \Delta u = f(u). 
\end{equation}
In this way one can use techniques of nonlinear elliptic equations to prove a certain type of symmetry. It is worth pointing out that, in general, the stream function need not to be a solution of a problem like~\eqref{semilinear}, but the two questions are intimately related in the analytic case, as has been recently shown in~\cite{EHSX24}. 

In this paper we are concerned with entire solutions with \emph{bounded energy}, that is, $v\in L^2(\RD)$. This is a quite natural assumption since the energy is a physically relevant quantity in fluid mechanics; in particular, it is an invariant for classical solutions to the Euler equations. Nonetheless, to the best of our knowledge, this setting has not been considered in the literature on steady solutions to~\eqref{Eu}.

Our first result is as follows:
\begin{theorem}
	\label{main:thm0}
	Let $v\in L^2(\RD) \cap \mathcal{C}^1(\RD)$ be a solution of the Euler equations~\eqref{Eu} and let $u$ be the corresponding stream function. Then:
	\begin{enumerate}[label=(A\arabic*)]
		\item\label{A1} $v$ has at least a stagnation point, i.e.~the set $\mathbf S$ is not empty;
		\item\label{A2} if $\mathbf S \neq \R^2$ is connected, up to consider $-u$ and to sum a constant, there holds:
		\begin{enumerate}
			\item[i)] $\displaystyle u(x) =1 = \max_{\R^2} u\ $ for all $x \in \mathbf S$;
			\item[ii)] $\displaystyle\lim_{|x| \to +\infty} u(x)=\inf_{\R^2}u=L \in [-\infty, 1)$; 
			\item[iii)]	$u$ solves
			\[
			-\Delta u = f(u) \quad \text{in } \RD,
			\]
			for some $f \in \mathcal C^0((L,1])$.
		\end{enumerate}
	\end{enumerate}
\end{theorem}

In words, we first show that there must be at least a stagnation point. Moreover, if the stagnation set is connected, then $u$ has a limit at infinity and it solves a semilinear elliptic equation~\eqref{semilinear}. At this point, one is tempted to conjecture that the stream function must be radially symmetric. We can prove this under a certain control on the oscillations of the vorticity at infinity:

\begin{theorem}
	\label{main:thm}
	Let $v\in L^2(\RD) \cap \mathcal{C}^1(\RD)$ be a solution of~\eqref{Eu} for which the stagnation set $\mathbf S$ is connected. Assume that there exists $B(R)$, ball of radius $R>0$, such that
	\begin{equation}
		\label{H1}\tag{H} \omega  \text{ has no local extrema in} \  \RD \setminus B(R).
	\end{equation} 
	Then $v$ is a circular flow or, in other words, the stream function $u$ is radially symmetric with respect to a point. In particular, the stagnation set $\mathbf S$ is either a point or a closed ball. 
	
	Moreover, $\omega$, $p$ and $B$ belong to $L^1(\R^2)$ (for a suitable choice of the constant in the definition of $p$) and
	\begin{gather*}
		\lim_{|x| \to + \infty} |v(x)| =  \lim_{|x| \to + \infty} \omega(x) = \lim_{|x| \to + \infty} B(x) =0,\\
		\int_{\R^2} \omega(x) \, dx =\int_{\R^2} B(x) \, dx = 0.
	\end{gather*}	
\end{theorem}

It is worth pointing out that the connectedness of the stagnation set $\mathbf S$ is a necessary assumption for Theorems~\ref{main:thm0},~\ref{main:thm} to hold. Indeed, in~\cite{MPW12} the authors construct a finite energy nonradial solution to a semilinear equation in the form \eqref{semilinear}. In particular, this construction gives rise to a solution to~\eqref{Eu}. Moreover, one can check that this solution also satisfies~\eqref{H1}, but its stagnation set is disconnected. Another example is~\cite{EFR24}, where a solution to~\eqref{Eu} with compact support and noncircular streamlines is built. In this case, the associated stream function is not given by an autonomous semilinear equation like~\eqref{semilinear}.

Let us compare Theorems~\ref{main:thm0},~\ref{main:thm} with the most related results in literature. In~\cite{GPSY21} the vorticity is assumed to be nonnegative, which is actually incompatible with $v \in L^2(\RD)$ (it suffices to take $\phi=1$ in Lemma~\ref{lemma-u-phi}). In~\cite{CFZ24,HN23} radial symmetry has been shown in the case of exterior domains $\RD \setminus \Omega$ under tangential boundary conditions ($v \cdot \nu=0$ on $\partial \Omega$) or no slip boundary conditions ($v =0$ on $\partial \Omega$). In both cases, $v$ is assumed to be of class $ \mathcal{C}^2$. More importantly, there is more information on the set $\Omega$: it is the closure of a bounded regular domain, or it is just a point, or a ball. We emphasize that here we only require the set $\mathbf S$ to be connected, and no assumptions are made on $v$, $\omega$ on $\partial \mathbf S$.

Another difference is that in~\cite{CFZ24, HN23} it is required
\[
\liminf_{|x| \to +\infty} |v(x)| > 0,
\]
which is incompatible with $v \in L^2(\RD)$. From this point of view, our settings are complementary. Finally, it is assumed that the radial component of $v$ satisfies $v_{\mathrm{rad}} =  o(1/|x|)$ as $|x| \to +\infty$. This implies that the streamlines are close to circles near infinity, which enables the use of the moving plane technique from infinity. Our condition at infinity is of different nature, as we only need to avoid oscillations of $\omega$ near infinity. As explained below, we do not use a moving plane procedure in our arguments.

We now provide some insights into the proofs of these results. Assertion~\ref{A1} in Theorem~\ref{main:thm0} is based on a the following fact: if $\mathbf S$ is compact or connected, then the stream function $u$ has limit at infinity (see Proposition~\ref{prop:L2grad}). The proof of this result follows from some estimates on the oscillations of $u$ on large spheres, and this is the starting point of the rest of the paper. 

With this in hand, we are able to show that $u$ solves a semilinear equation if $\mathbf S$ is connected. Because of our low regularity framework we can not directly argue as in~\cite{HN17, HN18, HN19, HN23}; our proof is somewhat different and uses that the Bernoulli function is constant along the streamlines by~\eqref{Eu2}.

In Theorem~\ref{main:thm} we show radial symmetry of the stream function. It is worth pointing out that the usual moving plane technique is not applicable here. First, because of the lack of regularity of the function $f(u)$ (which is only continuous), but more importantly because of the lack of information on the set $\mathbf{S}$. In this low regularity framework, we prove symmetry by using the continuous Steiner symmetrization. This technique, developed in~\cite{Bro95,Bro00}, has attracted much attention recently, see~\cite{CHVY19, GPSY21, Ru23, WZ23}. Roughly speaking, given a function $u$, the idea is to build a 1-parameter family of functions $u^t$ ($t \in [0, +\infty]$) such that $u^0=u$ and $u^{\infty}$ is the Steiner symmetrization of $u$. 

This method is based on the identity:
\begin{equation} \label{identity} 
	\int_{\R^2} \nabla u \cdot \nabla (u^t-u) \, dx = \int_{\R^2} f(u)(u^t-u) \, dx,
\end{equation} 
which is formally obtained by multiplying~\eqref{semilinear} by $u^t-u$ and integrating by parts. With some work one can prove that:
\begin{equation} \label{otro} 
	\liminf_{t \to 0^+} \frac 1 t \int_{\R^2} f(u)(u^t-u) \, dx \geq 0.
\end{equation}
But this estimate on the left hand side of~\eqref{identity} implies local radial symmetry, see~\cite{Bro00}.

In our case, this integration by parts has to be justified, and the integrability of the functions $f(u)u$, $f(u)u^t$ has to be proved. Recall that the limit $L$ of $u$ at infinity could be $-\infty$: moreover, a priori, the function $f$ could not be defined in $L$. Finally, the proof of~\eqref{otro} requires also that $F(u)$ is integrable, where $F$ is a primitive of $f$. We are able to show the integrability of those functions thanks to assumption~\eqref{H1}, together with suitable estimates in enlarging balls (see Lemmas~\ref{lemma-u-phi},~\ref{FL1}). Finiteness of the energy is a crucial ingredient in many parts of the proofs, as well as assumption~\eqref{H1}.

Still, the symmetry results proved in~\cite{Bro00} cannot be applied directly to our framework. One of the reasons is that the limit $L$ can be $- \infty$, a case not covered in~\cite{Bro00}. This forces us to adapt the argument of continuous Steiner symmetrization to our setting. In particular we need adaptations of the Pólya-Szegő inequality as well as local symmetry results (Propositions~\ref{lemma:polya},~\ref{prop:loc:symm}).

The rest of the paper is organized as follows. In Section~\ref{sec:prel} we prove Theorem~\ref{main:thm0}. Section~\ref{sec:brock} is devoted to present the basic concepts of the continuous Steiner symmetrization and to adapt some of these results to our framework. In Section~\ref{sec:radialsymmetry}, we establish a radial symmetry result for solutions to elliptic semilinear problems, from which Theorem~\ref{main:thm} follows in Section~\ref{sec:teo2}.

\medskip

\subsection*{Notation} We use the notation $B_{q}(R)$ to denote the open ball centered at $q$ with radius $R>0$, and $A_{q}(R_1, R_2)$ ($0 \leq R_1 < R_2 \leq +\infty$) to denote the corresponding open annulus. If the center $q$ is $0$ we drop the subscript and write directly  $B(R)$, $A(R_1, R_2)$. For any measurable set $\Omega$, we denote by $|\Omega|$ its Lebesgue measure.

\section{Proof of Theorem~\ref{main:thm0}} \label{sec:prel}

We begin this section by proving general geometric properties of a function over $\R^N$ with gradient in $L^N(\R^N)$, $N \geq 2$, that, in our opinion, are of independent interest.

\begin{proposition}
\label{prop:L2grad}
Consider $u\in\mathcal C^1(\R^N)$ with $ \nabla u \in L^N(\R^N)$ and set
\[
\mathbf S=\{x\in\R^N:\nabla u(x)= 0 \}.
\]
Assume one of the following: 
\begin{enumerate}
\item[(i)] $\mathbf S$ is compact,
\item[(ii)] $\mathbf S$ is connected and $u\in\mathcal C^N(\R^N)$.
\end{enumerate}
Then 
 there exists $L\in[-\infty,+\infty]$ such that
\[
\lim_{\abs{x}\to+\infty}u(x)=L.
\]
\end{proposition}

\begin{remark} 
First we observe that the limit $L$ above could be infinite. This is the case of the function $\psi: \R^N \to \R$,
\[
\psi (x) = \left (\log(2+|x|^2) \right )^{\alpha}, \quad  \alpha \in (0,1-1/N).
\]
Clearly $\psi \in \mathcal{C}^\infty(\RN)$ and it is easy to check that $\nabla \psi \in L^N(\R^N)$.

Moreover, the assumption on the set $\mathbf S$ is necessary for the existence of a limit at infinity. Indeed, the function.
\[
u: \R^N \to \R, \quad u(x)= \cos (\psi(x)),
\]
satisfies that $\nabla u \in L^N(\R^N)$ but has no limit at infinity. In this case,
\[
\mathbf S =  \bigcup_{k=1}^{+\infty} \partial B(R_k) \cup \{0\}, \quad R_k = \sqrt{e^{(k\pi)^{1/\alpha}} -2 }.
\]

\end{remark}

\medskip 
In order to prove Proposition~\ref{prop:L2grad} we need the following lemma.
\begin{lemma}\label{lem:max} Let $u \in \mathcal C^{1}(\R^N)$ with $ \nabla u \in L^N(\R^N)$. Then there exists a sequence $(R_n)_{n \in \N} \subseteq (0,+\infty)$, $R_n \to +\infty$ such that
\[
\max \{ |u(x) - u(y)|: \ x, y \in \partial B(R_n) \} \to 0, \quad \text{as } n \to +\infty.
\]
\end{lemma}

\begin{proof}
	
We start by observing that any function $\phi \in W^{1,N}(\mathbb{S}^{N-1})$ is continuous and there holds:
\begin{equation} \label{morrey} 
	\max \{ |\phi(x) - \phi(y)|: \ x, y \in \mathbb{S}^{N-1} \} \leq C \| \nabla \phi \|_{L^N(\mathbb{S}^{N-1})}. 
\end{equation}
Assume, reasoning by contradiction, that there exist $\e>0$ and $R_0>0$ such that
\[
\max \{ |u(x) - u(y)|: \ x, y \in \partial B(R) \} \geq \e, \quad \mbox{for all } R > R_0.
\]
For any such $R$, define $\phi_R: \mathbb{S}^{N-1} \to \R$, $\phi_R (z)= u(Rz)$. By~\eqref{morrey}, we have:
\begin{align*}
 \e^N &\leq \max \{ |u(x) - u(y)|^N: \ x, y \in \partial B(R) \} \\
 	&= \max \{ |\phi_R(x) - \phi_R(y)|^N: \ x, y \in \mathbb{S}^{N-1} \} \\
	&\leq C \int_{\mathbb{S}^{N-1}} |R \, \nabla^{T}u(Rz )|^N \, d\sigma(z) \\
	&\leq  C R \int_{\partial B(R)}  |\nabla u(x) |^N \, d\sigma(x). 
 \end{align*}
Here $\nabla^T$ denotes the tangential derivative in $\mathbb{S}^{N-1}$. Then,
\[
\int_{\R^N} |\nabla u(x)|^N \, dx \geq \int_{R_0}^{+\infty} \int_{\partial B(r)}  |\nabla u(x) |^N \, d\sigma(x)  \, dr \geq \frac{\e^N}{C} \int_{R_0}^{+\infty} \frac{1}{r} \, dr = + \infty,
\]
a contradiction which finishes the proof. 

\end{proof}

\begin{proof}[Proof of Proposition~\ref{prop:L2grad}]
	
Define:
\[
\overline{L} = \limsup_{|x| \to +\infty} u(x) \in [-\infty, +\infty], \quad \text{and} \quad  \underline{L} = \liminf_{|x| \to + \infty} u(x) \in [-\infty, +\infty].
\]

Reasoning by contradiction, let us assume that $\overline{L} > \underline{L}$. 
Take $(R_n)_{n \in \N}$ as given by Lemma~\ref{lem:max}: up to a subsequence, we can assume that $u(x_n) \to \hat{L} \in [\underline{L}, \overline{L}]$ for $x_n \in \partial B(R_n)$. By the definition of $R_n$, the limit does not depend on the choice of the point $x_n \in \partial B(R_n)$. Observe that $\hat{L} $ cannot coincide with both $\underline{L}$ and $ \overline{L}$: in what follows we assume that $\hat{L} \neq \overline{L}$, the other case being analogous.

By the definition of $\overline{L}$, there exists a sequence  $(y_n)_{n \in \N} \subseteq \R^N$, $|y_n| > R_n$ such that $u(y_n) \to \overline{L}$ as $n \to +\infty$. We define $\sigma(n)$ so that $y _n \in B(R_{\sigma(n)})$. We consider $u$ in the annulus $A_n = A(R_n, R_{\sigma(n)})$, and observe that:
\[
u_{|\partial A_n} = \hat{L} + o_n(1), \quad y_n \in A_n, \quad u(y_n) = \overline{L} + o_n(1).
\]

As a consequence the function $u_{|A_n}$ attains its maximum at some $z_n \in A_n$. In sum, there exists a sequence $z_n$ such that
\begin{equation} 
	\label{zn} |z_n| \to +\infty, \quad \nabla u(z_n)=0, \quad u(z_n) \to \overline{L}. 
\end{equation}
 
In particular $z_n \in \mathbf{S}$ and this is clearly a contradiction with (i). Let us discuss now the case (ii). Observe that by continuity $u({\mathbf S})$ is an interval. Moreover, by  Sard's Theorem,  $u({\mathbf S})$ has zero Lebesgue measure in $\R$, and this implies that $u$ is constant in $\mathbf S$. By~\eqref{zn}, $u \equiv \overline{L}$ on $\mathbf S$.

Since  $\mathbf S$ is unbounded and connected, $\mathbf S \cap \partial B(R_n) \neq \emptyset$ if $n$ is sufficiently large. This implies that $u(z)= \hat{L}$ for any $z \in \mathbf S$, and this gives the desired contradiction. 

\end{proof}

\begin{corollary}
\label{cor:L2grad} 
Let $u\in\mathcal C^1(\R^N)$ be such that $\nabla u\in L^N(\R^N)$.
Then there exists $x\in\R^N$ such that $\nabla u(x)=0$.
\end{corollary}

\begin{proof}
Assume by contradiction that $u$ has no critical points, then we can apply Proposition~\ref{prop:L2grad} to deduce that
\[
\lim_{\abs{x}\to+\infty}u(x)=L.
\]
If $u\equiv L$, we are done. Otherwise, either $\sup u >L$ or $\inf u < L$:  then $u$ must achieve a maximum or a minimum, respectively: a contradiction occurs.

\end{proof}

\begin{proof}[Proof of Theorem~\ref{main:thm0}]

 We recall that the stream function $u$ is defined as the unique function $u:\RD\to\R$ (up to a constant) such that $(\nabla u)^\perp=(-\partial_{x_2}u,\partial_{x_1}u)=v$.

Since $v\in\mathcal C^1(\R^2)$, we have $u\in\mathcal C^2(\R^2)$, so we are in the conditions of Proposition~\ref{prop:L2grad}. The statement~\ref{A1} follows immediately from  Corollary~\ref{cor:L2grad} above, so we just need to prove~\ref{A2}.

We can argue as in the last part of the proof of Proposition~\ref{prop:L2grad}: $u(\mathbf S)$ is an interval in $\R$ by continuity, but it has zero Lebesgue measure by Sard's Theorem, so that $u$ is constant in $\mathbf S$. Since there are no other critical points, it turns out that the values $L$ and $u(\mathbf S)$, correspond to the infimum and supremum of $u$, which must be different values. Up to changing sign and adding a constant, we can assume that $u \equiv 1$ in $\mathbf{S}$, and $L = \inf u <1$. The statements i) and ii) follow.
	
We now turn our attention to iii), which is proved in three steps.

\medskip

\noindent {\bf Step 1.  For any $c \in (L,1)$ the level sets $\Gamma_c= \{ x \in \R^2: u(x) =c \}$ are $\mathcal C^2$ closed and connected curves containing $\mathbf S$ in their interior.}

\medskip

Given $c \in (L,1)$, the set $\Gamma_c$ is $\mathcal C^2$ by the implicit function theorem. Moreover, $\Gamma_c$ is compact since $\lim_{|x| \to \infty} u(x)=L$. As a consequence, there exists a countable set $I$ such that 
	\[
	\Gamma_c = \bigcup_{i \in I} \gamma_i,
	\]
	where $\gamma_i$ are closed $\mathcal C^2$ curves in $\R^2$. 
	
 We denote by $I(\gamma_i)$ the interior of $\gamma_i$, that is, the bounded connected component of $\R^2 \setminus \gamma_i$. Since $u$ is constant on $\gamma_i$ we conclude that there is at least a local maximum or minimum of $u$ in $I(\gamma_i)$; since all critical points form the connected set $\mathbf S$, we conclude $\mathbf S \subseteq I(\gamma_i)$. 
	
	We now prove that the indices set $I$ has only one element, that is, $\Gamma_c$ is connected. Otherwise assume that $\gamma_1$ and $\gamma_2$ are closed curves in $\Gamma_c$. By the previous argument, $\mathbf S \subseteq I(\gamma_1) \cap I(\gamma_2)$. Without loss of generality, we can assume that $\gamma_2 \subseteq I(\gamma_1)$. Then, again, $u$ would have a local maximum or minimum in $I(\gamma_1) \setminus \overline{I(\gamma_2)}$, but $\mathbf S\cap\left( I(\gamma_1) \setminus \overline{I(\gamma_2)}\right)=\emptyset$, a contradiction. This concludes the proof of Step 1.
 
 \medskip 

\noindent {\bf Step 2. There exists a continuous function $f:(L,1)\to \R$ such that $-\Delta u = f(u)$ in $\R^2 \setminus \mathbf S$.}

\medskip 	

	For any $c \in (L,1)$, take $x \in \Gamma_c$ and define $F(c)= -B(x)$, where $B$ is the Bernoulli function $B= \frac 1 2 |\nabla u|^2 +p$. Because of~\eqref{Eu2}, $F$ is well defined as $B(x)$ does not depend on the choice of the point $x \in \Gamma_c$. We now claim that the function $F$ is $\mathcal C^1$ in $(L,1)$.
	
	Let $c \in (L, 1)$ and $x \in \Gamma_c$. For some $\delta >0$, define $ \sigma: (-\delta, \delta) \to \R^2$ as the solution of
	\[
	\begin{cases}  \sigma'(t)= \nabla u (\sigma(t)), & t \in (-\delta, \delta),\\
	\sigma(0)=x. &  \end{cases} 
	\]
	Since $\nabla u(x) \not= 0$, we have that $g'(0)>0$, where $g= u \circ \sigma$. By taking a smaller $\delta>0$ if necessary and suitable $\e_1,\e_2 >0$, we have that $g: (-\delta, \delta) \to (c-\e_1, c + \e_2 )$ is a $\mathcal C^1$ diffeomorphism. Hence we can write the function $F$ as:
	\[
	F_{|(c-\e_1, c+\e_2)}= - B \circ \sigma \circ g^{-1}.
	\]
	Since $B$ is of class $\mathcal C^1$ in $\Omega$, then $F_{|(c-\e_1, c+\e_2)}$ is a $\mathcal C^1$ function.
	
Again by~\eqref{Eu2}, we have:
\[\omega \nabla u = \nabla B = -F'(u) \nabla u  \quad \mbox{in } \RD \setminus \mathbf{S},
\]
and this implies that $- \Delta u = - \omega = f(u)$ in $\RD \setminus \mathbf{S}$, where $f(t)=F'(t)$.
	
\medskip 

\noindent {\bf Step 3. The function $f(t)$ can be continuously extended to $t=1$ and $-\Delta u = f(u)$ in $\R^2$.}

\medskip 	
	
Take $x \in \R^2 \setminus \mathbf S$, and $y \in \mathbf S$ minimizing the distance to $x$. In this way, the segment $[x,y]$ intersects $\mathbf S$ only at $y$. Take any sequence $(t_n)_{n\in\N}\subseteq(L,1)$ such that $t_n\to1$ as $n\to+\infty$. Then, for sufficiently large $n\in\N$, there exists $y_n\in [x,y]$ such that $u(y_n)=t_n$. Any accumulation point of $(y_n)_{n \in \N}$ must belong to $[x,y]$ (because of compactness) and to $\mathbf S$ (since $u(y_n) \to 1)$. Then, $y_n\to y \in\mathbf S$ as $n\to+\infty$. By the continuity of $\Delta u$, we have
		\[
		-\Delta u(y)=-\lim_{n\to+\infty}\Delta u(y_n)=\lim_{n\to+\infty} f(t_n).
		\]
Then we can define $f(1)= -\Delta u(y)$, and $f$ is continuous at $t=1$ by definition. 

In order to conclude the proof of this step we need to show that $-\Delta u \equiv f(1)$ in $\mathbf S$. Given $z\in \partial \mathbf S$  there exists $(z_n)_{n\in\N}\subseteq\R^2\setminus\mathbf S$ such that $z_n\to z$ as $n\to+\infty$. By continuity, $u(z_n)=s_n \to 1$ as $n\rightarrow +\infty$. By continuity of $f$ and $\Delta u$, we have: 
\[
-\Delta u(z) \leftarrow -\Delta u(z_n) = f(u(z_n)) = f(s_n) \to f(1).
\]
		
This already shows the result if $\mathring{\mathbf{S}} = \emptyset$ (here we denote by $\mathring{A}$ the interior of a set $A$).	Instead, if $\mathring{\mathbf{S}} \not=\emptyset$, then for all $z\in\mathring{\mathbf{S}}$ it is immediate to show that $-\Delta u(z)=0$. Moreover, take $\bar z\in  \partial \mathring{\mathbf S}$:  by continuity of $\Delta u$ we conclude $\Delta u(\bar z)=0$. But $\bar z \in \partial \mathbf S$ and  by the above argument $-\Delta u(\bar{z}) = f(1)$. Hence, if $\mathring{\mathbf{S}} \not=\emptyset$ then $f(1)=0$ and the conclusion holds.

\end{proof}

\begin{remark} There are many papers in the literature where a key point is to pass from solutions to~\eqref{Eu} to semilinear elliptic problems, see for instance~\cite{CFZ24, EHSX24, HN17, HN18, HN19, HN23, Ru23,  WZ23}. As far as we know, all those require at least $\mathcal C^2$ regularity on the vector field $v$. Instead, we have been able to prove it for $\mathcal C^1$ regularity of the vector field; the key idea for this is to use the Bernoulli function rather than the vorticity.

Moreover, we only assume connectedness of the stagnation set, contrarily to many other works in the field. In particular, the argument above could be applied in the framework of~\cite{Ru23}: the statement of Theorem A there holds for $\Omega = \Omega_ 0 \setminus \mathbf{S}$, where $\Omega$ is a $\mathcal C^0$ simply connected bounded domain and $\mathbf{S} \subseteq \Omega_0$ is a connected set.
\end{remark}

\section{Continuous Steiner symmetrization}\label{sec:brock}

In this section we recall the definition and basic properties of the continuous Steiner symmetrization introduced by Brock in~\cite{Bro95,Bro00}.

\begin{definition}[see Theorem 2.1 in~\cite{Bro00} and Definition 2.10 in~\cite{CHVY19}]
For any measurable set $M$ in $\R$ with finite measure, we define its {\em continuous Steiner symmetrization} $M^t$ for any $t \geq 0$, as follows.
\begin{enumerate}
\item If the measurable
set is an open interval of $\R$, i.e.~$I= (a,b)$ with $a,b \in \R$, $a<b$, we define 
\[
I^t = (a^t, b^t),\]
with 
\[
a^t = \frac{1}{2}\left[a-b + e^{-t}(a+b)\right] \quad \text{and} \quad 
b^t = \frac{1}{2}\left[b-a + e^{-t}(a+b)\right].
\]
\item If $M = \bigcup_{k=1}^{m} I_k$, where $I_k = (a_k, b_k)$ are disjoint bounded intervals, we define
\[
M^t = \bigcup_{k=1}^{m} I_k^t,
\]
where $0 \leq t < t_1$, and $t_1$ is the first value such that two intervals $I_{k_1}^{t_1}$ and $I_{k_2}^{t_1}$ share a common endpoint. When this happens, we merge these intervals into a single open interval and repeat the process.

\item If $M = \bigcup_{k=1}^{+\infty} I_k$, where $I_k = (a_k, b_k)$ are disjoint bounded intervals, we define $M_j = \bigcup_{k=1}^{j} I_k$ for all $j \in \mathbb{N}$, and
\[
M^t = \bigcup_{j=1}^{+\infty} M_j^t.
\]

\item Let $M$ be any measurable set with $|M| < +\infty$. Then by the continuity from above of the Lebesgue measure, there exists a sequence of open sets $(A_k)_{k\in\N}\subseteq\R$ such that $M = \bigcap_{k=1}^{+\infty} A_k$ and $A_k \supseteq A_{k+1}$. Then
\[
M^t = \bigcap_{k=1}^{+\infty} A_k^t.
\]
\end{enumerate}
\end{definition}

\begin{remark} For any measurable set $M$ in $\R$ with finite measure, its continuous Steiner symmetrization $M^t$ satisfies the following properties for any $0\leq t \leq +\infty$, see~\cite[Theorem 8]{Bro95}:
    \begin{enumerate}
        \item $|M^t|=|M|$ (equimeasurabilty);
        \item If $M \subseteq N$, then $M^t \subseteq N^t$ (monotonicity);
        \item $(M^t)^s=M^{t+s}$ (semigroup property);
        \item $M^0=M$ and $M^\infty=M^*$, where $M^*$ is the Steiner symmetrization of $M$, that is
        \[
        M^*=\left(-\frac12\abs{M},\frac12\abs{M}\right).
        \]
    \end{enumerate}
\end{remark}
We can now define the continuous Steiner symmetrization of a measurable set $M\subseteq\R^N$  with respect to the direction $x_1$.
\begin{definition}
Given any measurable set $M\subseteq\R^N$ with $|M| < +\infty$, we define its {\em continuous Steiner symmetrization} in the direction $x_1$, as follows
\[
M^t = \left\{ x=(x_1,x') \in \mathbb{R}^N : x_1 \in [M(x')]^t, \, x' \in \mathbb{R}^{N-1} \right\},
\]
where $M(x') = \left\{ x_1 \in \mathbb{R} : x=(x_1,x') \in M \right\}$. Analogously one can define the continuous Steiner symmetrization with respect to any vector $\eta \in \R^N$, $|\eta | =1$. 
\end{definition}

To define the continuous Steiner symmetrization of a function, we restrict ourselves to the following class of functions:
\[
\mathcal{S}(\mathbb{R}^N) = \left\{ u: \mathbb{R}^N \to \mathbb{R}\,: \, u \text{ measurable}, \, \left| \{u > c\} \right| < +\infty, \, \forall c > \inf u \right\},
\]
where we write $\{u > c\}=\{x\in\R^N:u(x)>c\}$.
Note that $\inf u=-\infty$ is allowed. We also define:
\[
 \mathcal{S}_+(\mathbb{R}^N) = \left\{ u \in \mathcal{S}(\mathbb{R}^N): \ u \geq 0 \right \}.
 \]

\begin{definition}
   For any $u \in \mathcal{S}(\mathbb{R}^N)$ we define its {\em continuous Steiner symmetrization} $u^t$, for any $t \geq 0$, with respect to any direction $\eta \in \R^N$, $|\eta | =1$, as follows:
\[
u^t(x) = 
\begin{cases}
\sup \{ c : x \in \{u > c\}^t \} & \text{if } x \in \bigcup_{c > \inf u} \{u > c\}^t \\
\inf u & \text{if } x \not\in \bigcup_{c > \inf u} \{u > c\}^t.
\end{cases}
\]
\end{definition}

\begin{remark} The following properties hold true:
\begin{enumerate}
    \item $u^0=u$ and $u^\infty=u^*$, where $u^*$ is the Steiner symmetrization of $u$ with respect to $x_1$.
    \item If $u$ is continuous, then $u^t$ is continuous, see~\cite[Theorem 11]{Bro95}.
    \item If $u$ is Lipschitz continuous with constant $\mathrm{Lip}(u)$, then $u^t$ is Lipschitz with $\mathrm{Lip}(u^t)\le\mathrm{Lip}(u)$, see~\cite[Theorem 11]{Bro95}. 
\end{enumerate}
\end{remark}

In the following proposition, we collect some properties (slightly generalized) of~\cite{Bro95} about continuous Steiner symmetrization of functions that will be useful in the sequel.

\begin{proposition}
\label{prop:steiner}
Let $u,v \in \mathcal{S}(\mathbb{R}^N)$, $c \in \mathbb{R}$. Then the following properties hold:
\begin{enumerate}
\item $(u + c)^t = u^t + c$;
\item $\{u > c\}^t = \{u^t > c\}$;
\item if $u\le v$, then $u^t\le v^t$ (monotonicity);
\item if $F$ is a continuous function and $F(u)\in L^1(\R^N)$, then $F(u^t) \in L^1(\R^N)$ and
    \[
    \int_{\R^N}F(u)\,dx=\int_{\R^N}F(u^t)\,dx \quad\text{(Cavalieri's principle)};
    \]
\item if $G$ is a continuous non-decreasing function, then $[G(u)]^t = G(u^t)$.
\end{enumerate}
Note that since we are considering just measurable functions, the preceding properties are intended to hold up to a set of zero Lebesgue measure.
\end{proposition}
\begin{proof}
   Assertions (1),(2), and (3) are proved in~\cite[Theorem 5]{Bro95}, while (4) and (5) are proved in~\cite[Theorem~8]{Bro95}. The only difference is that in~\cite[Theorem~8]{Bro95} the function $G$ is strictly increasing: a comment is in order here.

    If $G$ is continuous and non-decreasing, then for all $d < \sup \, G$, there exists $d'\in \R$ such that $G^{-1}((d,+\infty))=(d',+\infty)$. Then we can use property (2) to obtain:
    \[
    \{G(u)^t>d\}=\{G(u)>d\}^t=\{u>d'\}^t,
    \]
    On the other hand
    \[
    \{G(u^t)>d\}=\{u^t>d'\}=\{u>d'\}^t,
    \]
    and the claim follows taking into account that two functions are equal a.e.~if their superlevel sets coincide a.e..

\end{proof}

An important property of many types of rearrangements is the so-called Pólya-Szegő inequality. This is contained in~\cite[Theorem 3.2]{Bro00} for $u \in H^1(\R^N) \cap \mathcal S_+(\R^N)$; however, we will need it in a more general setting, as stated below in Proposition~\ref{lemma:polya}.

Some preliminaries are in order. 
In what follows, given any $m \in \R$, we define:
\begin{equation}\label{def:Gn} 
	G_m(s)= \max \{ (s-m), 0\}, \ \ H_m(s)= \min\{s,m\}.
\end{equation}
Observe that both functions are continuous and non-decreasing.

\begin{lemma} \label{new}
Given any function $u \in H^1_{\mathrm{loc}}(\R^N) \cap \mathcal S(\R^N)$, $m \in (\inf \, u, \sup \, u )$, we have that:

\begin{enumerate} 
	\item  $G_m(u) \in H_{loc}^1(\R^N) \cap \mathcal S_+(\R^N)$ and $H_m(u) \in H_{loc}^1(\R^N) \cap \mathcal S(\R^N)$;
	\item $u= G_m(u) + H_m(u)$;
	\item If $\nabla u \in L^2(\R^N)$, then $G_m(u) \in H^1(\R^N)$;
	\item $G_m(u^t)= G_m(u)^t$ and $H_m(u^t)= H_m(u)^t$;
	\item $ |\nabla G_m(u) | \cdot |\nabla H_m(u) |=0$,  a.e.~in $\R^N$. 
\end{enumerate}
\end{lemma}

\begin{proof}
	Statements (1) and (2) are  immediate. Regarding (3), since $u  \in H^1_{\mathrm{loc}}(\R^N)$, then we can consider its quasi-continuous representative for which  its superlevel sets are quasi-open. It is clear that $\nabla G_m(u) \in L^2(\R^N)$. Moreover, $G_m(u)=0$ on $\partial \{ u >m\}$ in the trace sense. Recall that the set $\{ u >m\}$ has finite measure: by Poincaré inequality (see for instance \cite[Lemma 4.5.3]{HPbook}) $G_m(u) \in L^2(\R^N)$.
	
	The assertion (4) follows immediately from property (5) in Proposition~\ref{prop:steiner}. Finally, statement (5) follows since $\nabla u=0$ a.e.~in the set $\{u =m\}$. 
	
\end{proof}

We are now ready to prove our version of the Pólya-Szegő inequality for the continuous Steiner symmetrization:

\begin{proposition}
\label{lemma:polya}
    Let $u\in H^1_{\mathrm{loc}}(\R^N) \cap \mathcal S(\R^N)$ with $\nabla u\in L^2(\R^N)$. Then $u^t\in H^1_{\mathrm{loc}}(\R^N) \cap \mathcal S(\R^N)$ and
    \[
    \int_{\R^N}\abs{\nabla u}^2\,dx\ge\int_{\R^N}\abs{\nabla u^t}^2\,dx.
    \]
\end{proposition}

\begin{proof}
    Take $m \in (\inf \, u, \sup \, u)$ arbitrary. Observe that since $G_m(u) \in H^1(\R^N) \cap \mathcal S_+(\R^N)$, we can use  \cite[Theorem 3.2]{Bro00} to obtain:
    \[
    \int_{\R^N}\abs{\nabla G_m(u)}^2\,dx \geq \int_{\R^N}\abs{\nabla G_m(u)^t}^2\,dx.
    \]
    Taking into account property (5) in Lemma~\ref{new}, we have: 
    \begin{align*}
    \int_{\R^N}\abs{\nabla u}^2\,dx& =\int_{\R^N}\abs{\nabla (G_m(u) + H_m(u))}^2\,dx\\
    &   =\int_{\R^N} (\abs{\nabla G_m(u)}^2 + \abs{\nabla H_m(u)}^2)\,dx\\
    & \geq  \int_{\R^N}\abs{\nabla G_m(u)}^2\,dx \\
    &\geq \int_{\R^N}\abs{\nabla G_m(u)^t}^2\,dx = \int_{\R^N}\abs{\nabla G_m(u^t)}^2\,dx.
    \end{align*}
    
  We conclude since, by the monotone convergence theorem, we have
      \[
  \int_{\R^N}\abs{\nabla G_m(u^t)}^2\,dx \to \int_{\R^N}\abs{\nabla u^t}^2\,dx,\quad\text{as }m\to \inf  u.
    \]
       
\end{proof}

In general, an important question in symmetric rearrangements is the symmetry implications of the equality case in the Pólya-Szegő inequality. In the case of the continuous Steiner symmetrization we cannot conclude radial symmetry because of the possible presence of plateaus. This motivates the following definition:

\begin{definition}[see Theorem 6.1 in~\cite{Bro00}] \label{locallyrad}
	We say that a function $u\in\mathcal C^1(\R^2)$ is \emph{locally symmetric in any direction} if
	\[
	\R^N\setminus\mathbf S=\bigcup_{k\in K}A_k,
	\]
	where
	\begin{enumerate}
		\item $\mathbf S=\{x\in\R^2:\nabla u(x)=0\}$,
		\item $K$ is countable,
		\item $A_k$ are disjoint open annuli $A_{q_k}(r_k,R_k)$, ($0 \leq r_k < R_k \leq +\infty$) and $u$ is radially symmetric and strictly radially decreasing in each domain $A_k$.
	\end{enumerate}
\end{definition}

The following result is an adaptation to our setting of \cite[Theorem 6.2]{Bro00}, where it is proved for positive functions in a Sobolev Space. This proposition will be the essential tool to prove the radial symmetry result of Theorem~\ref{main:thm}.

\begin{proposition}
\label{prop:loc:symm}
Let $u \in  \mathcal{C}^1(\R^N) \cap  \mathcal S(\R^N)$ with $\nabla u \in L^2(\R^N)$. Assume that, for any direction $\eta \in \R^{N}$, $|\eta|=1$, the corresponding continuous Steiner symmetrized function $u^t$ satisfies:
	\begin{equation} \label{limite}
	 \lim_{t \to 0^+} \frac{1}{t} \displaystyle \int_{\R^N} (|\nabla u(x)|^2 - |\nabla u^t(x)|^2)\, dx =0.
	 \end{equation}
Then $u$ is locally symmetric in any direction.	

\end{proposition}

\begin{proof}

Take $m \in (\inf \, u, \sup \, u)$ and decompose $u= G_m(u) + H_m(u)$. Observe that $G_m(u)$ is of class $\mathcal{C}^1$ in the set $\{m < u < \sup u\}$, so that it satisfies the assumptions of \cite[Theorem 6.2]{Bro00}.

As in the proof of Proposition~\ref{lemma:polya}, we can decompose:
\[
\begin{split}
&\int_{\R^N} \big( \abs{\nabla u}^2-\abs{\nabla u^t}^2\big) \,dx \\ 
& \quad = \int_{\R^N} \big (\abs{\nabla G_m(u)}^2 + \abs{\nabla H_m(u)}^2-\abs{\nabla G_m(u^t)}^2 - \abs{\nabla H_m(u^t)}^2 \big ) \, dx\\
&\quad =\int_{\R^N} \big (\abs{\nabla G_m(u)}^2 - \abs{\nabla G_m(u)^t}^2 \big ) \, dx + \int_{\R^N} \big ( \abs{\nabla H_m(u)}^2 - \abs{\nabla H_m(u)^t}^2 \big ) \, dx\\ 
&\quad \geq  \int_{\R^N} \big (\abs{\nabla G_m(u)}^2 - \abs{\nabla G_m(u)^t}^2 \big ) \, dx, 
\end{split}
\]
where the last inequality is due to Proposition~\ref{lemma:polya} applied to the function $H_m(u)$. Then, by applying such result to $G_m(u)$ and using~\eqref{limite}, we obtain:
\[ 
0 \leq \frac 1 t \int_{\R^N} \big (\abs{\nabla G_m(u)}^2 - \abs{\nabla G_m(u)^t}^2 \big ) \, dx \leq \frac{1}{t} \displaystyle \int_{\R^N} (|\nabla u(x)|^2 - |\nabla u^t(x)|^2)\, dx \to 0,
\]
as $t \to 0^+$. Then we can apply \cite[Theorem 6.2]{Bro00} to $G_m(u)$ to deduce that $G_m(u)$ is locally symmetric in any direction. Since $m$ is arbitrary, we conclude that $u$ is locally symmetric in any direction.

\end{proof}

We end up this section by recalling a quantitative estimate of the $L^2$ continuity of the Steiner symmetrization, valid in bounded domains, see \cite[Theorem 4.2]{Bro00}.

\begin{proposition}
	\label{prop:l2}
	Let $u \in H_0^1(B(R))$ be a nonnegative function, and $u^t$ its continuous Steiner symmetrization in the direction $x_1$. Then,
	\[
	\|u^t -u \|_{L^2(B(R))} \leq t R \| \partial_{x_1} u \|_{L^2(B(R))}.
	\]
\end{proposition}

\section{Radial symmetry for semilinear PDEs} \label{sec:radialsymmetry}
In this section we prove a symmetry result for semilinear elliptic equations, which will be used in the proof of Theorem~\ref{main:thm}. Even if some of the results below can be extended to higher dimension, we prefer to restrict ourselves to the case $N=2$, which is the scope of  Theorem~\ref{main:thm}.

The following is the main result of this section.

\begin{theorem}
    \label{thm:rad}
    Let  $u \in  H^2_{\mathrm{loc}}(\R^2)$ with $\nabla u\in L^2(\R^2)$ such that:
    \[
    \lim_{\abs{x}\to+\infty}u(x)=L\in[-\infty, \ + \infty),
    \]
    with  $u(x)>L$ for all $x \in \R^2$. Assume that 
    \begin{equation}\label{eq:weaksemilihelpeq}
    	-\Delta u=f(u),\quad\text{in }\R^2,
    \end{equation}
    where $f\in\mathcal C^0((L, \sup u])$ is monotone in an interval $(L, \LL)$, for some $\LL >L$. 
    
    Then $u$ is locally symmetric in any direction.  Moreover,
    \begin{enumerate}
    	\item[i)] $\displaystyle \lim_{t \to L} f(t)=0$, $f(u) \in L^1(\R^2)$ and
	\[
	\int_{\R^2} f(u) \, dx =0.
	\]
    	\item[ii)] $f$ is integrable in $(L, \sup u)$ and, defining $F(t)= \int_L^t f(s) \, ds$, we have that $F(u) \in L^1(\R^2)$ and 
	\[
	\int_{\R^2} F(u) \, dx =0.
	\]
    \end{enumerate}
\end{theorem}

Before proving Theorem~\ref{thm:rad}, we start with some preliminary results, which we believe are of independent interest.
\begin{lemma}
\label{lemma-u-phi}
    Let $u\in H^2_{\mathrm{loc}}(\R^2)$ with $\nabla u\in L^2(\R^2)$ and $\phi\in H^1_{\mathrm{loc}}(\R^2)$ with $\nabla\phi\in L^2(\R^2)$.
    Then for any sequence $(R_n)_{n\in\N}\subseteq (0,+\infty)$ such that $R_n \to +\infty$ and
\begin{equation}
\label{eq:lemma:nota}
R_n\log R_n\int_{\partial B(R_n)}\abs{\nabla u}^2\,d \sigma(x)\to0,\quad\text{as }n\to+\infty,
\end{equation}
one has
\[
\int_{ B(R_n)}(-\Delta u)\phi\,dx\to\int_{\RD}\nabla u\cdot\nabla\phi\,dx,\quad\text{as }n\to+\infty.
\]
\end{lemma}

\begin{remark} 
	Observe that such sequence $(R_n)_{n \in \N}$ always exists, since
	\[
	\int_{\R^2} |\nabla u|^2 \, dx = \int_0^{+\infty} \int_{\partial B(r)}  |\nabla u|^2 \, d \sigma(x) \, dr < + \infty,
	\]
	so that the map $r \mapsto \int_{\partial B_r}  |\nabla u|^2 \, d \sigma(x)$ is integrable in $(0,+\infty)$.
\end{remark}

\begin{proof}

We first claim that for any $\phi \in H^1_{loc}(\R^2)$ and for all $r_2>r_1>0$, it holds
\begin{equation}
    \label{med:phi2}
    \begin{split}
    \left | \left(\frac{1}{r_2}\int_{\partial B(r_2) } \abs{\phi}^2\,d\sigma(x)\right)^{1/2} \right. & \left.  -\left(\frac{1}{r_1}\int_{\partial B(r_1)} \abs{\phi}^2\,d\sigma(x)\right)^{1/2} \right | \\
        &\le\left(\log(r_2/r_1)\int_{A(r_1,r_2)} \abs{\nabla \phi}^2\,dx\right)^{1/2}.
    \end{split}
\end{equation}

By density arguments, and using the trace inequality in $A(r_1, r_2)$, we can assume that $\phi $ is a $\mathcal C^\infty$ function. 

For all $r \in (r_1, r_2)$, we denote the mean (up to a constant) of $\abs{\phi}^2$ on $\partial B(r)$ as $\bar \phi: [r_1, r_2 ]\to \R$,
\[
\bar\phi(r)=\frac{1}{r}\int_{\partial B(r)} \abs{\phi(x)}^2\,d\sigma(x)=\int_{\partial B(1)} \abs{\phi(ry)}^2\,d\sigma(y).
\]

Observe that if $\phi$ is $\mathcal C^\infty$, then $\bar \phi $ is $\mathcal C^\infty$. Differentiating under the integral sign and using Hölder's inequality, we have:
\begin{align*}
\abs{\bar\phi'(r)}=\left|\frac{d\bar\phi(r)}{dr}\right|&=2\left|\int_{\partial B(1)}{\phi(ry)}\nabla\phi(ry)\cdot y\,d\sigma(y)\right| \\    
    &\le 2\left(\int_{\partial B(1)} \abs{\phi(ry)}^2\,d\sigma(y)\right)^{1/2}\left(\int_{\partial B(1)} \abs{\nabla\phi(ry)}^2\,d\sigma(y)\right)^{1/2}\\
    &=2(\bar\phi(r))^{1/2}\left(\int_{\partial B(r)} \abs{\nabla\phi(x)}^2\,d\sigma(x)\right)^{1/2} \cdot r^{-1/2}.
\end{align*}
Define $A= \{r \in (r_1, r_2): \bar \phi(r) > 0\}$. For $r \in A$ we can divide both members above by $2(\bar\phi(r))^{1/2}$ and integrate:
\begin{align*}
\left|\int_{A}\frac{\bar\phi'(r)}{2(\bar\phi(r))^{1/2}}\,dr\right|&\le\int_{A}\frac{\abs{\bar\phi'(r)}}{2(\bar\phi(r))^{1/2}}\,dr\\
    &\le\int_{r_1}^{r_2}\left(\int_{\partial B(r)} \abs{\nabla\phi(x)}^2\,d\sigma(x)\right)^{1/2}r^{-1/2}\,dr\\
    &\le\left(\int_{r_1}^{r_2}\int_{\partial B(r)} \abs{\nabla\phi(x)}^2\,d\sigma(x)dr\right)^{1/2}\left(\int_{r_1}^{r_2}\frac1r\,dr\right)^{1/2}\\
    &=\left(\log(r_2/r_1)\int_{A(r_1, r_2)} \abs{\nabla \phi}^2\,dx\right)^{1/2}.
\end{align*}
On the other hand, by the Barrow's rule,
\[
\left|\int_{A}\frac{\bar\phi'(r)}{2(\bar\phi(r))^{1/2}}\,dr\right|=\left | (\bar\phi(r_2))^{1/2}-(\bar\phi(r_1))^{1/2} \right |,
\]
and the claim is proved.

\medskip 

In particular,~\eqref{med:phi2} with $r_2=R_n$ and $r_1=1$ gives
\begin{equation}
    \label{stima:palla}
\int_{\partial B(R_n)}\abs{\phi}^2\,d\sigma(x)\le C R_n\log R_n,
\end{equation}
where $C$ is a positive constant depending only on $\phi$.

Integrating $-\Delta u $ multiplied by $\phi$ over $ B(R_n)$, one has
\begin{equation}
\label{eq:lapl}
    \int_{ B(R_n)}(-\Delta u)\phi\,dx=\int_{ B(R_n)}\nabla u\cdot\nabla\phi\,dx-\int_{\partial B(R_n)} \phi \, \partial_\nu u\,d\sigma(x).
\end{equation}
By Hölder's inequality and~\eqref{stima:palla}
\[
\left|\int_{\partial B(R_n)} \phi \, \partial_\nu u\,d\sigma(x)\right|\le \left(CR_n\log R_n\int_{\partial B(R_n)}\abs{\nabla u}^2\,d\sigma(x)\right)^{1/2}\to 0,
\]
as $n\to+\infty$ by~\eqref{eq:lemma:nota}.

Finally, by the dominated convergence theorem
\[
\int_{ B(R_n)}\nabla u\cdot\nabla\phi\,dx\to\int_{\R^2}\nabla u\cdot\nabla\phi\,dx,\quad\text{as }n\to+\infty,
\]
and the claim follows, passing to the limit in~\eqref{eq:lapl}.

\end{proof}

The proof of the following lemma is based on a Pohozaev-type argument, where the boundary terms need to be treated with care.

\begin{lemma}
\label{FL1}
Let  $u \in H^2_{\mathrm{loc}}(\R^2)$ with $\nabla u\in L^2(\R^2)$  be a solution of 
\[
-\Delta u=f(u), \quad\text{in }\R^2,
\]
for some continuous function $f$. Then there exists a choice of $F$, primitive of $f$, such that
\[
\int_{B(R)}F(u)\,dx\to0,\quad\text{as }R\to+\infty.
\]
\end{lemma}

\begin{proof} 
Since $u \in H^2_{\mathrm{loc}}(\R^2)$, by standard local regularity estimates, we conclude that $u$ is locally $\mathcal C^{1, \alpha}$. Let $F$ be any primitive of $f$, we will fix a specific one later. Now, integrating over $ B(R)$ the equation $-\Delta u =f(u)$ multiplied by $\nabla u\cdot x$ one has
\begin{align*}
\int_{B(R)}-\Delta u(\nabla u\cdot x)\,dx&=\int_{B(R)}f(u)(\nabla u\cdot x)\,dx\\
	&=\int_{B(R)}\nabla(F(u))\cdot x\,dx\\
	&=\int_{\partial B(R)}F(u)x\cdot\nu\,d\sigma(x)-2\int_{B(R)}F(u)\,dx\\
	&=R\int_{\partial B(R)}F(u)\,d\sigma(x)-2\int_{B(R)}F(u)\,dx.
\end{align*}
On the other hand
\[
\Delta u(\nabla u\cdot x)=\mathrm{div}\left((\nabla u\cdot x)\nabla u-\frac12\abs{\nabla u}^2x\right).
\]
Thus
\begin{align*}
\int_{B(R)}-\Delta u(\nabla u\cdot x)\,dx&=-\int_{\partial B(R)}(\nabla u\cdot x)\partial_\nu u\,d\sigma(x)+\frac12\int_{\partial B(R)}\abs{\nabla u}^2x\cdot\nu\,d\sigma(x)\\
	&=-R\int_{\partial B(R)}\abs{\partial_\nu u}^2\,d\sigma(x)+\frac R2\int_{\partial B(R)}\abs{\nabla u}^2\,d\sigma(x)\\
	&=-\frac R2 \int_{\partial B(R)}\left(\abs{\partial_\nu u}^2-\abs{\partial_\theta u}^2\right)\,d\sigma(x),
\end{align*}
where $\partial_\theta u$ denotes the tangential derivative of $u$ on $\partial B(R)$.
Then
\begin{align*}
2\int_{B(R)}F(u)\,dx=R\int_{\partial B(R)}F(u)\,d\sigma(x)+\frac R2 \int_{\partial B(R)}\left(\abs{\partial_\nu u}^2-\abs{\partial_\theta u}^2\right)\,d\sigma(x).
\end{align*}
Let us set
\[
\Psi(R)=\int_{B(R)}F(u)\,dx, \quad  h(R)=\frac12 \int_{\partial B(R)}\left(\abs{\partial_\nu u}^2-\abs{\partial_\theta u}^2\right)\,d\sigma(x).
\]
Note that $h\in L^1(0,+\infty)$.
Taking into account that
\[
\Psi'(R)=\int_{\partial B(R)}F(u)\,d\sigma(x),
\]
one has
\[
\left[\frac{\Psi(R)}{R^2}\right]'=\frac{\Psi'(R)R-2\Psi(R)}{R^3}=-\frac{h(R)}{R^2},
\]
which is integrable in $(1,+\infty)$ and then
\[
\frac{\Psi(R)}{R^2}\to \ell,\quad\text{as }R\to+\infty,
\]
for some $\ell\in \R$. We now add an appropriate  constant to the function $F$ so that $\ell=0$ and this allows us to write
\begin{align*}
\abs{\Psi(R)}&=R^2\left|\frac{\Psi(R)}{R^2}\right|\\
	&=R^2\left|\int_R^{+\infty}\left[\frac{\Psi(r)}{r^2}\right]'\,dr\right|\\
	&\le R^2\int_R^{+\infty}\frac{\abs{h(r)}}{r^2}\,dr\\
	&\le\int_R^{+\infty}\abs{h(r)}\,dr\to0,\quad\text{as }R\to+\infty,
\end{align*}
since $h\in L^1(0,+\infty)$.

\end{proof}

We are now ready to prove the main result of this section, Theorem~\ref{thm:rad}.

\bigskip

\begin{proof}[Proof of Theorem~\ref{thm:rad}]

Since $u \in H^2_{\mathrm{loc}}(\R^2)$, by standard local regularity estimates, we conclude that $u$ is locally $\mathcal C^{1, \alpha}$. The proof is developed in several steps.

\medskip

\noindent {\bf Step 1. Proof of statement (i).}

\medskip

By hypothesis, $f$ is monotone in $(L, \LL)$, which implies that there exists
 \[
 \lim_{t \to L} f(t) = \beta \in [-\infty, +\infty].
 \]
 By Lemma~\ref{lemma-u-phi} with $\phi\equiv1$ one has
 \begin{equation}
 	\label{eq:2342134531}
 	\int_{ B(R_n)}f(u)\,dx\to0,\quad\text{as }n \to+\infty,
 \end{equation}
 for some $R_n\to+\infty$ as $n \to+\infty$. This immediately implies that $ \beta =0$.
 
\medskip
 
We claim that $f$ is nonpositive and non-increasing in $(L, \LL)$. In order to prove it, let us recall that $\nabla u \in L^2(\R^2)$ and $u \in \mathcal C^{1, \alpha}_{\mathrm{loc}}(\R^2)$ weakly solves~\eqref{eq:weaksemilihelpeq}, i.e.
\[
\int_{\R^2} \nabla u \cdot \nabla \varphi \, dx = \int_{\R^2} f(u) \varphi \, dx, \qquad \text{for all } \varphi \in \mathcal C^{1}_c(\R^2).
\]
Now, take $c >L$ arbitrary; for every $\varepsilon >0$ let $g_\varepsilon \in \mathcal C^\infty(\R)$ be such that $g_\varepsilon(t)=1$ if $t \geq c$, $g_\varepsilon(t)=0$ if $t \leq c-\varepsilon$, and $g_\varepsilon'(t) \geq 0$ for any $t \in \R$.
Let us consider $\varphi=g_\varepsilon(u)$ that belongs to $\mathcal C^{1}_c(\R^2)$ if $\varepsilon$ is sufficiently small. Hence, plugging it as test function and taking into account the monotonicity of the function $g_\varepsilon$ we obtain:
\[
0 \leq \int_{\R^2} |\nabla u|^2g'_\varepsilon(u) \, dx = \int_{\R^2} f(u) g_\varepsilon(u) \, dx.
\]
By the dominated convergence theorem, we conclude that
\[
0 \leq \int_{\{u \geq c\}} f(u) \,dx, \quad \text{for any } c>L.
\]

Fix now $c \in (L, \LL)$. Taking into account~\eqref{eq:2342134531} and the monotonicity of $f$, we conclude that $f(t)\leq 0$ for all $ t \in (L, c)$, and this concludes the proof of the claim.
 
Since $f$ is nonpositive in $(L, \LL)$, we have that $f(u(x)) \leq 0 $ for all $x \in \R^2 \setminus B(R)$, for sufficiently large $R>0$. Again by~\eqref{eq:2342134531} we conclude that $f(u) \in L^1(\R^2)$ and
\[
\int_{\R^2} f(u) \, dx =0.
\]
 
\medskip

\noindent {\bf Step 2. Proof of statement (ii).}
 
 \medskip
 
 We can argue analogously with $F(u)$: indeed, since $f(t) \leq 0$ if $t \in(L, \LL)$ then $F(t)$ is non-increasing in $(L, \LL)$ and admits a limit as $t \to L$. We choose a constant in the definition of $F(t)$ so that Lemma~\ref{FL1} holds, that is,
 \begin{equation}
 \label{otramas} \int_{B(R)} F(u) \, dx \to 0, \quad \mbox{as } R \to +\infty.
 \end{equation}
 As a consequence, $ \lim_{t \to L} F(t)=0$. As $f$ is nonpositive in $(L, \LL)$, we conclude that $F$ is nonpositive and integrable in $(L, \LL)$, and $F(t)= \int_L^t f(s) \, ds$. In particular, $F(u(x))$ is nonpositive in $\R^2\setminus B(R)$; by~\eqref{otramas}, $F(u)\in L^1(\R^2)$ and 
 \[
 \int_{\R^2} F(u) \, dx =0.
 \]
 
\medskip

\noindent {\bf Step 3. Proof of the local symmetry in every direction.}

\medskip

In this step we make use of the continuous Steiner symmetrization, which  was introduced in Section~\ref{sec:brock}. First, some preliminary observations are in order.

Fix $\eta \in \R^2$, $|\eta|=1$, and denote by $u^t$ the continuous Steiner symmetrization of $u$ in the direction $\eta$. Applying Lemma~\ref{lemma-u-phi} with $\phi=u$, respectively $\phi=u^t$, we have that 
\[
 \int_{B(R_n)} f(u) u \, dx \to \int_{\R^N} |\nabla u|^2 \, dx,  \quad \text{and} \quad  \int_{B(R_n)} f(u) u^t \, dx \to \int_{\R^N} \nabla u \cdot \nabla u^t \, dx,
 \]
for some sequence $R_n \to + \infty$ (note that $\nabla u^t \in L^2(\RD)$ by Proposition~\ref{lemma:polya}). Observe that we also have
\[
 \lim_{|x| \to + \infty} u^t(x) =L, \quad \text{and} \quad  u^t > L  \quad \text{in }  \R^2.
 \]
As a consequence the functions $u$, $u^t$ have constant sign in the exterior of a ball $B(R)$. In particular, the functions $f(u)u$, $f(u)u^t$ have constant sign in the exterior of a ball, and hence they belong to $L^1(\R^2)$.\\
\\
Now, for any $t\ge0$, we set:
\[
J(t)=\int_{\R^2}f(u)(u^t-u)\,dx.
\]
We start by showing the following claim, where we control the right derivative at $0$ of the functional $J$ in terms of $t$.\\

\noindent\emph{Claim: there holds
	\begin{equation}
		\label{derJ}
		\liminf_{t\to0^+}\frac{J(t)}{t}\ge0.
\end{equation}}\\
Since $F(u)\in L^1(\R^2)$, by Cavalieri's principle, (4) in Proposition~\ref{prop:steiner}, we have that $F(u^t) \in L^1(\R^2)$ and
\[
0=\int_{\R^2}\left(F(u^t)-F(u)\right)\,dx=\int_{\R^2}\left(\int_0^1f(u+s(u^t-u))(u^t-u)\,ds\right)\,dx.
\]
It is important to observe that, a priori, it is not clear if the expression $f(u+s(u^t-u))(u^t-u)$ is Lebesgue integrable in the variables $x$, $s$. Still, the expression above is well defined. Then,
\begin{align*}
J(t)&=\int_{\R^2}f(u)(u^t-u)\,dx\\
&=\int_{\R^2}\left(\int_0^1\left(f(u)-f(u+s(u^t-u))\right)(u^t-u)\,ds\right)\,dx=J_1(t)+J_2(t),
\end{align*}
where
\begin{gather*}
J_1(t)=\int_{B(R)}\left(\int_0^1\left(f(u)-f(u+s(u^t-u))\right)(u^t-u)\,ds\right)\,dx,\\
J_2(t)=\int_{\R^2\setminus B(R)}\left(\int_0^1\left(f(u)-f(u+s(u^t-u))\right)(u^t-u)\,ds\right)\,dx,
\end{gather*}
and $R>0$ is such that $u< \LL$ in $\R^2\setminus B(R)$. Then, we can write $u \leq \chi$, where $\chi$ is defined:
\[
\chi(x)= \left \{ \begin{array}{ll} \sup \ u & x \in B(R), \\ \LL & x \in  \R^2\setminus B(R).  \end{array} \right.
\]
By the monotonicity property (3) in Proposition~\ref{prop:steiner}, one has $u^t\le \chi^t=\chi$, which implies that $u^t \leq \LL$ in $\R^2\setminus B(R)$.
Thus, since $f$ is non-increasing in $(L, \LL)$ as shown in Step 1, we immediately get that
\[
\left(f(u)-f(u+s(u^t-u))\right)(u^t-u)\geq 0,\quad\text{in }(\RD \setminus B(R)) \times (0,1).
\]
In particular, this fact implies $J_2(t)\ge0$ (a priori $J_2(t)$ could be $+\infty$).

\medskip 
Hence, to prove claim~\eqref{derJ}, it is enough to show that
\[
\lim_{t\to0^+}\frac{J_1(t)}{t}=0.
\]
By boundedness we have $\left(f(u)-f(u+s(u^t-u))\right)(u^t-u)\in L^1(B(R) \times (0,1))$. Thus, by Fubini Theorem and H\"older's inequality we have
\begin{align*}
J_1(t)&= \int_0^1 \left(\int_{B(R)}\left(f(u)-f(u+s(u^t-u))\right)(u^t-u)\,ds\right)\,dx\\
&\le\int_0^1\|u^t-u\|_{L^2(B(R))}\|f(u)-f(u+s(u^t-u))\|_{L^2(B(R))}\,ds.
\end{align*}
Now, we take $m\in (L,\min_{B(R)}u)$. Observe that  
\[
u^t - u = G_m(u^t) - G_m(u) = G_m(u)^t - G_m(u) \ \mbox{ in } B(R),
\]
where $G_m$ is defined in~\eqref{def:Gn}.

Moreover, for $\bar{R}>R$ suitably large, $G_m(u) \in H^1_{0}( B(\bar{R}))$ and is nonnegative. We can use Proposition~\ref{prop:l2} to get
\begin{align*}
\|u^t-u\|_{L^2(B(R))}&=\|G_m(u)^t - G_m(u)\|_{L^2( B(R))}\\
    &\le\|G_m(u)^t - G_m(u)\|_{L^2( B(\bar{R}))}\\
	&\le t\,\bar{R}\|\nabla u\|_{L^2(\R^2)}.
\end{align*}

\medskip 

Since $u^t\to u$ in $L^2(B(R))$, we have also convergence a.e.~up to subsequences and since $f$ is continuous
\[
\|f(u)-f(u+s(u^t-u))\|_{L^2(B(R))}\to0,\quad\text{as }t\to0^+,
\]
for any $s\in(0,1)$. The claim~\eqref{derJ} follows.

\medskip

To conclude the proof, we apply Lemma~\ref{lemma-u-phi} with $\phi=u^t-u$ to get
\[
\int_{ B(R_n)}f(u)(u^t-u)\,dx\to\int_{\R^2}\nabla u\cdot\nabla(u^t-u)\,dx,\quad\text{as }n\to+\infty,
\]
and being $f(u)u, f(u)u^t\in L^1(\R^N)$ we can pass to the limit
\[
\int_{\R^2}\nabla u\cdot\nabla(u^t-u)\,dx=\int_{\R^2}f(u)(u^t-u)\,dx=J(t).
\]
By convexity
\[
J(t)=\int_{\R^2}\nabla u\cdot\nabla(u^t-u)\,dx\le \frac 1 2 \int_{\R^2}\big(\abs{\nabla u^t}^2-\abs{\nabla u}^2\big)\,dx,
\]
and together with Proposition~\ref{lemma:polya} and claim~\eqref{derJ} we deduce
\[
0\ge\lim_{t\to0^+}\frac{\displaystyle \int_{\R^2}\big(\abs{\nabla u^t}^2-\abs{\nabla u}^2\big)\,dx}{ 2t}\ge\liminf_{t\to0^+}\frac{J(t)}{t}\ge0.
\]
Since the direction $\eta$ is arbitrary we can use Proposition~\ref{prop:loc:symm} to conclude that $u$ is locally symmetric in any direction.

\end{proof}

\section{ Proof of Theorem~\ref{main:thm}} \label{sec:teo2}

The proof of Theorem~\ref{main:thm} readily follows from Theorem~\ref{thm:rad}. Indeed, observe that ${\bf S} \neq \R^2$ because of hypothesis~\eqref{H1}. Then we are in conditions to apply Theorem~\ref{main:thm0}. 

Observe that $ \omega = - f(u)$: by~\eqref{H1} there exists $\LL>L$ such that $f(t) $ is monotone in $(L, \LL)$. We are then under the hypotheses of Theorem~\ref{thm:rad} and hence $u$ is locally symmetric in any direction, in the sense of Definition~\ref{locallyrad}, that is,
\[
\R^2\setminus\mathbf S=\bigcup_{k\in K}A_k,
\]
where
\begin{enumerate}
	\item $K$ is countable,
	\item $A_k$ are disjoint open annuli $A_{q_k}(r_k,R_k)$, ($0 \leq r_k < R_k \leq +\infty$) and $u$ is radially symmetric and strictly radially decreasing in each domain $A_k$.
\end{enumerate}	

In particular, $\partial A_k \in \mathbf S$ for all $k$. This implies in particular that $R_k = +\infty$ (otherwise we have two connected components of $\partial A_k$). Since the annuli are disjoint, we only have one annulus, that is, $\R^2 = A \cup \mathbf S$, with $A=A_{q}(r, +\infty)$, for some $q \in \R^2$ and $r \geq 0$. 

Observe now that $\omega = - f(u)$ and $B = p + \frac 1 2 |{v }|^2 = -F(u)$. From the statements i) and ii) of Theorem~\ref{thm:rad} we conclude that $\omega$, $B$ and $p$ belong to $L^1(\R^2)$ (for a suitable choice of the constant). Moreover:
\[
\lim_{|x| \to + \infty} \omega(x)= \lim_{|x| \to + \infty} B(x)=0, \quad \text{and} \quad 
 \int_{\R^2} \omega(x) \, dx = \int_{\R^2} B(x) \, dx=0.
\]
We only need to show that 
\[ \lim_{|x| \to +\infty} |{v}(x)|= 0.
\] 
Observe that, up to a translation, $u$ is a radial function $u(r)= u(|x|)$ satisfying:
\[
u''(r) + u'(r)/r +f (u(r))=0.
\]
Multiplying by $u'(r)$ we then obtain that
\[
\left( \frac 1 2 u'(r)^2 + F(u(r))  \right)'= -u'(r)^2/r \leq 0.
\]
Hence,  in this case, the function $p(r) = -\frac 1 2 \, u'(r)^2 -F(u(r)) $ is non-decreasing in $r$, and as a consequence it has a limit at infinity. Since $\lim_{r \to + \infty} F(u(r))=0$ and $\nabla u \in L^2(\R^2)$, we conclude that $\lim_{r \to + \infty} u'(r) =0$.

\bigskip

\subsection*{Data availability statement} This manuscript has no associated data.

\subsection*{Conflict of interest statement} The authors declare that there is no conflict of interest.

\bibliographystyle{abbrv}
\bibliography{DeR-E-R.bib}

\end{document}